\documentclass[12pt]{amsart}
\usepackage{amsmath,amsthm,amsfonts,amssymb,mathrsfs}
\usepackage{color}

\usepackage{tikz}

\usepackage{amssymb,cite}
\usepackage[colorlinks,plainpages,citecolor=magenta, linkcolor=blue, backref]{hyperref}

\usepackage{hyperref}
\date{\today}

\usepackage{hyperref}
\input{xy}
 \xyoption{all}
 \xyoption{arc}

\newtheorem{theorem}{Theorem}

\newtheorem{proposition}{Proposition}

\newtheorem{lemma}{Lemma}
\theoremstyle{definition}

\newtheorem{example}{Example}
\newtheorem{remark}{Remark}

\begin{document}

\title[On the semigroup of injective  endomorphisms of the semigroup $\boldsymbol{B}_{\omega}^{\mathscr{F}^3}$]{On the semigroup of injective monoid endomorphisms of the monoid $\boldsymbol{B}_{\omega}^{\mathscr{F}^3}$ with a three element family $\mathscr{F}^3$ of inductive nonempty subsets of $\omega$}
\author{Oleg Gutik and Marko Serivka}
\address{Ivan Franko National University of Lviv, Universytetska 1, Lviv, 79000, Ukraine}
\email{oleg.gutik@lnu.edu.ua, marko.serivka@lnu.edu.ua}

\keywords{Bicyclic monoid, inverse semigroup, bicyclic extension, endomorphism, semigroup of endomorphisms, multiplicative semigroup of positive integers.}

\subjclass[2020]{20M18, 20F29, 20M10.}

\begin{abstract}
We describe injective monoid endomorphisms of the semigroup $\boldsymbol{B}_{\omega}^{\mathscr{F}^3}$ with a three element family $\mathscr{F}^3$ of inductive nonempty subsets of $\omega$. Also, we show that the monoid $\boldsymbol{End}_*^1(\boldsymbol{B}_{\omega}^{\mathscr{F}})$ of all injective endomorphisms of the semigroup $\boldsymbol{B}_{\omega}^{\mathscr{F}^3}$ is isomorphic to the multiplicative semigroup of positive integers.
\end{abstract}

\maketitle


\section{\textbf{Introduction, motivation and main definitions}}

We shall follow the terminology of~\cite{Clifford-Preston-1961, Clifford-Preston-1967, Lawson=1998}. By $\omega$ we denote the set of all non-negative integers and by $\mathbb{N}$ the set of all positive integers.

\smallskip

Let $\mathscr{P}(\omega)$ be  the family of all subsets of $\omega$. For any $F\in\mathscr{P}(\omega)$ and any integer $n$ we put $n+F=\{n+k\colon k\in F\}$ if $F\neq\varnothing$ and $n+\varnothing=\varnothing$.
A subfamily $\mathscr{F}\subseteq\mathscr{P}(\omega)$ is called \emph{${\omega}$-closed} if $F_1\cap(-n+F_2)\in\mathscr{F}$ for all $n\in\omega$ and $F_1,F_2\in\mathscr{F}$. For any $a\in\omega$ we denote $[a)=\{x\in\omega\colon x\geqslant a\}$.

\smallskip

A subset $A$ of $\omega$ is said to be \emph{inductive}, if $i\in A$ implies $i+1\in A$. Obvious, that $\varnothing$ is an inductive subset of $\omega$.

\begin{remark}[\cite{Gutik-Mykhalenych=2021}]\label{remark-1.1}
\begin{enumerate}
  \item\label{remark-1.1(1)} By Lemma~6 from \cite{Gutik-Mykhalenych=2020} nonempty subset $F\subseteq \omega$ is inductive in $\omega$ if and only $(-1+F)\cap F=F$.
  \item\label{remark-1.1(2)} Since the set $\omega$ with the usual order is well-ordered, for any nonempty inductive subset $F$ in $\omega$ there exists nonnegative integer $n_F\in\omega$ such that $[n_F)=F$.
  \item\label{remark-1.1(3)} Statement \eqref{remark-1.1(2)} implies that the intersection of an arbitrary finite family of nonempty inductive subsets in $\omega$ is a nonempty inductive subset of  $\omega$.
\end{enumerate}
\end{remark}

For an arbitrary semigroup $S$ any homomorphism $\alpha\colon S\to S$ is called an \emph{endomorphism} of $S$. If the semigroup has the identity element $1_S$ then the endomorphism $\alpha$ of $S$ such that $(1_S)\alpha=1_S$ is said to be a \emph{monoid endomorphism} of $S$. A bijective endomorphism of $S$ is called an \emph{automorphism}.

\smallskip

A semigroup $S$ is called {\it inverse} if for any
element $x\in S$ there exists a unique $x^{-1}\in S$ such that
$xx^{-1}x=x$ and $x^{-1}xx^{-1}=x^{-1}$. The element $x^{-1}$ is
called the {\it inverse of} $x\in S$. If $S$ is an inverse
semigroup, then the function $\operatorname{inv}\colon S\to S$
which assigns to every element $x$ of $S$ its inverse element
$x^{-1}$ is called the {\it inversion}.

\smallskip


If $S$ is a semigroup, then we shall denote the subset of all
idempotents in $S$ by $E(S)$. If $S$ is an inverse semigroup, then
$E(S)$ is closed under multiplication and we shall refer to $E(S)$ as a
\emph{band} (or the \emph{band of} $S$). Then the semigroup
operation on $S$ determines the following partial order $\preccurlyeq$
on $E(S)$: 
\begin{center}
$e\preccurlyeq f$ if and only if $ef=fe=e$.
\end{center}
This order is
called the {\em natural partial order} on $E(S)$. 

\smallskip

If $S$ is an inverse semigroup then the semigroup operation on $S$ determines the following partial order $\preccurlyeq$
on $S$: $s\preccurlyeq t$ if and only if there exists $e\in E(S)$ such that $s=te$. This order is
called the {\em natural partial order} on $S$ \cite{Wagner-1952}.

\smallskip

The \emph{bicyclic monoid} ${\mathscr{C}}(p,q)$ is the semigroup with the identity $1$ generated by two elements $p$ and $q$ subjected only to the condition $pq=1$. The semigroup operation on ${\mathscr{C}}(p,q)$ is determined as
follows:
\begin{equation*}
    q^kp^l\cdot q^mp^n=q^{k+m-\min\{l,m\}}p^{l+n-\min\{l,m\}}.
\end{equation*}
It is well known that the bicyclic monoid ${\mathscr{C}}(p,q)$ is a bisimple (and hence simple) combinatorial $E$-unitary inverse semigroup and every non-trivial congruence on ${\mathscr{C}}(p,q)$ is a group congruence \cite{Clifford-Preston-1961}.

\smallskip

On the set $\boldsymbol{B}_{\omega}=\omega\times\omega$ we define the semigroup operation ``$\cdot$'' in the following way
\begin{equation}\label{eq-1.1}
  (i_1,j_1)\cdot(i_2,j_2)=
  \left\{
    \begin{array}{ll}
      (i_1-j_1+i_2,j_2), & \hbox{if~} j_1\leqslant i_2;\\
      (i_1,j_1-i_2+j_2), & \hbox{if~} j_1\geqslant i_2.
    \end{array}
  \right.
\end{equation}
It is well known that the bicyclic monoid $\mathscr{C}(p,q)$ is isomorphic to the semigroup $\boldsymbol{B}_{\omega}$ by the mapping $\mathfrak{h}\colon \mathscr{C}(p,q)\to \boldsymbol{B}_{\omega}$, $q^kp^l\mapsto (k,l)$ (see: \cite[Section~1.12]{Clifford-Preston-1961} or \cite[Exercise IV.1.11$(ii)$]{Petrich-1984}).

\smallskip

Next we shall describe the construction which is introduced in \cite{Gutik-Mykhalenych=2020}.

\smallskip

Let $\mathscr{F}$ be an ${\omega}$-closed subfamily of $\mathscr{P}(\omega)$. On the set $\boldsymbol{B}_{\omega}\times\mathscr{F}$ we define the semigroup operation ``$\cdot$'' in the following way
\begin{equation}\label{eq-1.2}
  (i_1,j_1,F_1)\cdot(i_2,j_2,F_2)=
  \left\{
    \begin{array}{ll}
      (i_1-j_1+i_2,j_2,(j_1-i_2+F_1)\cap F_2), & \hbox{if~} j_1\leqslant i_2;\\
      (i_1,j_1-i_2+j_2,F_1\cap (i_2-j_1+F_2)), & \hbox{if~} j_1\geqslant i_2.
    \end{array}
  \right.
\end{equation}
In \cite{Gutik-Mykhalenych=2020} is proved that if the family $\mathscr{F}\subseteq\mathscr{P}(\omega)$ is ${\omega}$-closed then $(\boldsymbol{B}_{\omega}\times\mathscr{F},\cdot)$ is a semigroup. Moreover, if an ${\omega}$-closed family  $\mathscr{F}\subseteq\mathscr{P}(\omega)$ contains the empty set $\varnothing$ then the set
\begin{equation*}
  \boldsymbol{I}=\{(i,j,\varnothing)\colon i,j\in\omega\}
\end{equation*}
is an ideal of the semigroup $(\boldsymbol{B}_{\omega}\times\mathscr{F},\cdot)$. For any ${\omega}$-closed family $\mathscr{F}\subseteq\mathscr{P}(\omega)$ the following semigroup
\begin{equation*}
  \boldsymbol{B}_{\omega}^{\mathscr{F}}=
\left\{
  \begin{array}{ll}
    (\boldsymbol{B}_{\omega}\times\mathscr{F},\cdot)/\boldsymbol{I}, & \hbox{if~} \varnothing\in\mathscr{F};\\
    (\boldsymbol{B}_{\omega}\times\mathscr{F},\cdot), & \hbox{if~} \varnothing\notin\mathscr{F}
  \end{array}
\right.
\end{equation*}
is defined in \cite{Gutik-Mykhalenych=2020}. The semigroup $\boldsymbol{B}_{\omega}^{\mathscr{F}}$ generalizes the bicyclic monoid and the countable semigroup of matrix units. In \cite{Gutik-Mykhalenych=2020} it is proven that $\boldsymbol{B}_{\omega}^{\mathscr{F}}$ is a combinatorial inverse semigroup and Green's relations, the natural partial order on $\boldsymbol{B}_{\omega}^{\mathscr{F}}$ and its set of idempotents are described.
Also, in \cite{Gutik-Mykhalenych=2020} the criteria when the semigroup $\boldsymbol{B}_{\omega}^{\mathscr{F}}$ is simple, $0$-simple, bisimple, $0$-bisimple, or it has the identity, are given.
In particularly in \cite{Gutik-Mykhalenych=2020} it is proven that the semigroup $\boldsymbol{B}_{\omega}^{\mathscr{F}}$ is isomorphic to the semigrpoup of ${\omega}{\times}{\omega}$-matrix units if and only if $\mathscr{F}$ consists of a singleton set and the empty set, and $\boldsymbol{B}_{\omega}^{\mathscr{F}}$ is isomorphic to the bicyclic monoid if and only if $\mathscr{F}$ consists of a non-empty inductive subset of $\omega$.

\smallskip

Group congruences on the semigroup  $\boldsymbol{B}_{\omega}^{\mathscr{F}}$ and its homomorphic retracts  in the case when an ${\omega}$-closed family $\mathscr{F}$ consists of inductive non-empty subsets of $\omega$ are studied in \cite{Gutik-Mykhalenych=2021}. It is proven that a congruence $\mathfrak{C}$ on $\boldsymbol{B}_{\omega}^{\mathscr{F}}$ is a group congruence if and only if its restriction on a subsemigroup of $\boldsymbol{B}_{\omega}^{\mathscr{F}}$, which is isomorphic to the bicyclic semigroup, is not the identity relation. Also in \cite{Gutik-Mykhalenych=2021}, all non-trivial homomorphic retracts and isomorphisms  of the semigroup $\boldsymbol{B}_{\omega}^{\mathscr{F}}$ are described. In \cite{Gutik-Mykhalenych=2022} it is proven that an injective endomorphism $\varepsilon$ of the semigroup $\boldsymbol{B}_{\omega}^{\mathscr{F}}$ is the indentity transformation if and only if  $\varepsilon$ has three distinct fixed points, which is equivalent to existence non-idempotent element $(i,j,[p))\in\boldsymbol{B}_{\omega}^{\mathscr{F}}$ such that  $(i,j,[p))\varepsilon=(i,j,[p))$.

\smallskip

In \cite{Gutik-Lysetska=2021, Lysetska=2020} the algebraic structure of the semigroup $\boldsymbol{B}_{\omega}^{\mathscr{F}}$ is established in the case when ${\omega}$-closed family $\mathscr{F}$ consists of atomic subsets of ${\omega}$. The structure of the semigroup $\boldsymbol{B}_{\omega}^{\mathscr{F}_n}$, for the family  $\mathscr{F}_n$ which is generated by the initial interval $\{0,1,\ldots,n\}$ of $\omega$, is studied in \cite{Gutik-Popadiuk=2023}. The semigroup of endomorphisms of $\boldsymbol{B}_{\omega}^{\mathscr{F}_n}$ is described in \cite{Gutik-Popadiuk=2022, Popadiuk=2022}.

\smallskip

In \cite{Gutik-Prokhorenkova-Sekh=2021} it is proven that the semigroup $\mathrm{\mathbf{End}}(\boldsymbol{B}_{\omega})$ of the endomorphisms of the bicyclic semigroup $\boldsymbol{B}_{\omega}$ is isomorphic to the semidirect products $(\omega,+)\rtimes_\varphi(\omega,*)$, where $+$ and $*$ are the usual addition and the usual multiplication on the set of non-negative integers $\omega$.

\smallskip

In the paper \cite{Gutik-Pozdniakova=2023}  injective endomorphisms of the semigroup $\boldsymbol{B}_{\omega}^{\mathscr{F}}$ with the two-elements family $\mathscr{F}$ of inductive nonempty subsets of $\omega$ are studies. Also, in \cite{Gutik-Pozdniakova=2023} the authors describe the elements of the semigroup $\boldsymbol{End}^1_*(\boldsymbol{B}_{\omega}^{\mathscr{F}})$ of all injective monoid endomorphisms of the monoid $\boldsymbol{B}_{\omega}^{\mathscr{F}}$, and show that Green's relations $\mathscr{R}$, $\mathscr{L}$, $\mathscr{H}$, $\mathscr{D}$, and $\mathscr{J}$  on $\boldsymbol{End}^1_*(\boldsymbol{B}_{\omega}^{\mathscr{F}})$ coincide with the relation of equality. In \cite{Gutik-Pozdniakova=2023a, Gutik-Pozdniakova=2023b} the semigroup $\boldsymbol{End}^1(\boldsymbol{B}_{\omega}^{\mathscr{F}})$ of all  monoid endomorphisms of the monoid $\boldsymbol{B}_{\omega}^{\mathscr{F}}$ is studied.

\smallskip

Later we assume that $\mathscr{F}^3$ is a family of inductive nonempty subsets of $\omega$ which consists of three sets.  By Proposition~1 of \cite{Gutik-Mykhalenych=2021} for any $\omega$-closed family $\mathscr{F}$ of inductive subsets in $\mathscr{P}(\omega)$ there exists an $\omega$-closed family $\mathscr{F}^*$ of inductive subsets in $\mathscr{P}(\omega)$ such that $[0)\in \mathscr{F}^*$ and the semigroups $\boldsymbol{B}_{\omega}^{\mathscr{F}}$ and $\boldsymbol{B}_{\omega}^{\mathscr{F}^*}$ are isomorphic. Hence without loss of generality we may assume that the family $\mathscr{F}$ contains the set $[0)$, i.e., $\mathscr{F}^3=\left\{[0),[1),[2)\right\}$. Later in the paper we denote $\mathscr{F}_{0,1}=\{[0),[1)\}$ and $\mathscr{F}_{1,2}=\{[1),[2)\}$ as subfamilies of $\mathscr{F}^3$.

\smallskip

In this paper we describe injective monoid endomorphisms of the semigroup $\boldsymbol{B}_{\omega}^{\mathscr{F}^3}$. Also, we show that the monoid $\boldsymbol{End}_*^1(\boldsymbol{B}_{\omega}^{\mathscr{F}})$ of all injective monoid endomorphisms of the semigroup $\boldsymbol{B}_{\omega}^{\mathscr{F}}$ is isomorphic to the multiplicative semigroup of positive integers.

\section{\textbf{Injective endomorphisms of the monoid $\boldsymbol{B}_{\omega}^{\mathscr{F}^3}$ are extensions of injective endomorphisms of its submonoid $\boldsymbol{B}_{\omega}^{\mathscr{F}_{0,1}}$}}\label{section-2}

If $\mathscr{F}$ is an arbitrary $\omega$-closed family $\mathscr{F}$ of inductive subsets in $\mathscr{P}(\omega)$ and $[s)\in \mathscr{F}$ for some $s\in \omega$ then
\begin{equation*}
  \boldsymbol{B}_{\omega}^{\{[s)\}}=\{(i,j,[s))\colon i,j\in\omega\}
\end{equation*}
is a subsemigroup of $\boldsymbol{B}_{\omega}^{\mathscr{F}}$ and by Proposition~3 of \cite{Gutik-Mykhalenych=2020} the semigroup $\boldsymbol{B}_{\omega}^{\{[s)\}}$ is isomorphic to the bicyclic semigroup.

\smallskip

Later we need the following theorem from \cite{Gutik-Mykhalenych=2022}.

\begin{theorem}[{\cite[Theorem~2]{Gutik-Mykhalenych=2022}}]\label{theorem-2.1}
Let $\mathscr{F}$ be an $\omega$-closed family of inductive nonempty subsets of $\omega$, which contains at least two sets. Then for an injective monoid endomorphism $\varepsilon$ of $\boldsymbol{B}_{\omega}^{\mathscr{F}}$ the following conditions are equivalent:
\begin{itemize}
  \item[$(i)$]   $\varepsilon$ is the identity map;
  \item[$(ii)$]  there exists a nonidempotent element $(i,j,[p))\in\boldsymbol{B}_{\omega}^{\mathscr{F}}$ such that $(i,j,[p))\varepsilon=(i,j,[p))$;
  \item[$(iii)$] the map $\varepsilon$ has at least three fixed points.
\end{itemize}
\end{theorem}

Let $\mathscr{F}^2=\{[0),[1)\}$. For an arbitrary positive integer $k$ and any $p\in\{0,\ldots,k-1\}$ we define the transformation $\alpha_{k,p}$  of the semigroup $\boldsymbol{B}_{\omega}^{\mathscr{F}^2}$ in the following way
\begin{align*}
  (i,j,[0))\alpha_{k,p}&=(ki,kj,[0)), \\
  (i,j,[1))\alpha_{k,p}&=(p+ki,p+kj,[1)),
\end{align*}
for all $i,j\in\omega$. Also, for an arbitrary positive integer $k\geqslant 2$ and any $p\in\{1,\ldots,k-1\}$ we define the transformation $\beta_{k,p}$  of the semigroup $\boldsymbol{B}_{\omega}^{\mathscr{F}^2}$ in the following way
\begin{align*}
  (i,j,[0))\beta_{k,p}&=(ki,kj,[0)), \\
  (i,j,[1))\beta_{k,p}&=(p+ki,p+kj,[0)),
\end{align*}
for all $i,j\in\omega$.

\smallskip

The following theorem is proved in \cite{Gutik-Pozdniakova=2023}.

\begin{theorem}[{\cite[Theorem~1]{Gutik-Pozdniakova=2023}}]\label{theorem-2.2}
Let $\mathscr{F}^2=\{[0),[1)\}$ and $\varepsilon$ be an injective monoid endomorphism of  $\boldsymbol{B}_{\omega}^{\mathscr{F}^2}$. Then either there exist a positive integer $k$ and $p\in\{0,\ldots,k-1\}$ such that $\varepsilon=\alpha_{k,p}$ or there exist a positive integer $k\geqslant 2$ and $p\in\{1,\ldots,k-1\}$ such that $\varepsilon=\beta_{k,p}$.
\end{theorem}

\begin{example}\label{example-2.3}
Let $\mathscr{F}^3=\{[0),[1),[2)\}$. Fix an arbitrary positive integer $k$. We define the transformation $\alpha_{[k]}$  of the semigroup $\boldsymbol{B}_{\omega}^{\mathscr{F}^3}$ in the following way
\begin{equation*}
  (i,j,[p))\alpha_{[k]}=
  \left\{
    \begin{array}{ll}
      (ki,kj,[p)), & \hbox{if~} p\in\{0,1\};\\
      (k(i+1)-1,k(j+1)-1,[2)), & \hbox{if~} p=2,
    \end{array}
  \right.
\end{equation*}
for all $i,j\in\omega$.
It is obvious that $\alpha_{[k]}$ is an injective transformation of the monoid $\boldsymbol{B}_{\omega}^{\mathscr{F}^3}$.
\end{example}

\begin{lemma}\label{lemma-2.4}
For an arbitrary positive integer $k$ the transformation $\alpha_{[k]}\colon \boldsymbol{B}_{\omega}^{\mathscr{F}^3}\to \boldsymbol{B}_{\omega}^{\mathscr{F}^3}$ is an injective monoid endomorphism of the semigroup $\boldsymbol{B}_{\omega}^{\mathscr{F}^3}$.
\end{lemma}

\begin{proof}
It is obvious that in the case when $k=1$ the map $\alpha_{[k]}$ is the identity transformation of the monoid $\boldsymbol{B}_{\omega}^{\mathscr{F}^3}$, i.e., $\alpha_{[k]}$ is an automorphism of $\boldsymbol{B}_{\omega}^{\mathscr{F}^3}$, and hence later without loss of generality we may assume that $k\geqslant 2$.

\smallskip

By Lemma~2 of \cite{Gutik-Pozdniakova=2023} the restrictions of the map $\alpha_{[k]}$ onto the subsemigroups $\boldsymbol{B}_{\omega}^{\mathscr{F}_{0,1}}$ and $\boldsymbol{B}_{\omega}^{\mathscr{F}_{1,2}}$ of $\boldsymbol{B}_{\omega}^{\mathscr{F}^3}$ are injective monoid endomorphism of  $\boldsymbol{B}_{\omega}^{\mathscr{F}_{0,1}}$ and $\boldsymbol{B}_{\omega}^{\mathscr{F}_{1,2}}$, respectively. Hence it is complete to show that the map $\alpha_{[k]}$ preserves the semigroup operation in the following two cases
\begin{equation*}
  (i_0,j_0,[0))\cdot(i_2,j_2,[2)) \qquad \hbox{and} \qquad (i_2,j_2,[2))\cdot(i_0,j_0,[0)).
\end{equation*}
We get that
\begin{align*}
  ((i_0,j_0,[0))&\cdot(i_2,j_2,[2)))\alpha_{[k]}=\\
  &=
  \left\{
    \begin{array}{ll}
      (i_0-j_0+i_2,j_2,(j_0-i_2+[0))\cap[2))\alpha_{[k]}, & \hbox{if~} j_0<i_2;\\
      (i_0,j_2,[0)\cap[2))\alpha_{[k]},                   & \hbox{if~} j_0=i_2;\\
      (i_0,j_0-i_2+j_2,[0)\cap(-1+[2)))\alpha_{[k]},      & \hbox{if~} j_0=i_2+1;\\
      (i_0,j_0-i_2+j_2,[0)\cap(i_2-j_0+[2)))\alpha_{[k]}, & \hbox{if~} j_0\geqslant i_2+2
    \end{array}
  \right.
\\
   &=
  \left\{
    \begin{array}{ll}
      (i_0-j_0+i_2,j_2,[2))\alpha_{[k]}, & \hbox{if~} j_0<i_2;\\
      (i_0,j_2,[2))\alpha_{[k]},         & \hbox{if~} j_0=i_2;\\
      (i_0,j_2+1,[1))\alpha_{[k]},       & \hbox{if~} j_0=i_2+1;\\
      (i_0,j_0-i_2+j_2,[0))\alpha_{[k]}, & \hbox{if~} j_0\geqslant i_2+2
    \end{array}
  \right.
\\
   &=
  \left\{
    \begin{array}{ll}
      (k(i_0-j_0+i_2+1)-1,k(j_2+1)-1,[2)), & \hbox{if~} j_0<i_2;\\
      (k(i_0+1)-1,k(j_2+1)-1,[2)),         & \hbox{if~} j_0=i_2;\\
      (ki_0,k(j_2+1),[1))),                & \hbox{if~} j_0=i_2+1;\\
      (ki_0,k(j_0-i_2+j_2),[0)),           & \hbox{if~} j_0\geqslant i_2+2,
    \end{array}
  \right.
\end{align*}
\begin{align*}
(i_0,&j_0,[0))\alpha_{[k]}\cdot(i_2,j_2,[2))\alpha_{[k]}= (ki_0,kj_0,[0))\cdot(k(i_2+1)-1,k(j_2+1)-1,[2))\\
   &=
  \left\{\!
    \begin{array}{l}
      (ki_0-kj_0+k(i_2+1)-1,k(j_2+1)-1,(kj_0-(k(i_2+1)-1)+[0))\cap[2)), \\
      \qquad \qquad \qquad \qquad \qquad \qquad \qquad \qquad \qquad \qquad \qquad \qquad \quad \hbox{if~} kj_0<k(i_2{+}1){-}1;\\
      (ki_0,k(j_2+1)-1,[0)\cap[2)), \qquad \qquad \qquad \qquad \qquad \qquad \; \, \hbox{if~} kj_0=k(i_2{+}1){-}1; \\
      (ki_0,kj_0-(k(i_2+1)-1)+k(j_2+1)-1,[0)\cap(k(i_2+1)-1-kj_0+[2))), \\
      \qquad \qquad \qquad \qquad \qquad \qquad \qquad \qquad \qquad \qquad \qquad \qquad \quad \hbox{if~} kj_0>k(i_2{+}1){-}1
    \end{array}
  \right.
  \\
   &=
  \left\{\!
    \begin{array}{ll}
      (k(i_0-j_0+i_2+1)-1,k(j_2+1)-1,[2)),                & \hbox{if~} j_0<i_2{+}1{-}1/k;\\
      (ki_0,k(j_2+1)-1,[2)),                              & \hbox{if~} j_0=i_2{+}1{-}1/k; \\
      (ki_0,k(j_0-i_2+j_2),[0)\cap(k(i_2+1)-1-kj_0+[2))), & \hbox{if~} j_0>i_2{+}1{-}1/k;
    \end{array}
  \right.
  \\
   &=
  \left\{\!
    \begin{array}{ll}
      (k(i_0-j_0+i_2+1)-1,k(j_2+1)-1,[2)), & \hbox{if~} j_0<i_2;\\
      (k(i_0+1)-1,k(j_2+1)-1,[2)),         & \hbox{if~} j_0=i_2;\\
      (ki_0,k(j_2+1),[1))),                & \hbox{if~} j_0=i_2+1;\\
      (ki_0,k(j_0-i_2+j_2),[0)),           & \hbox{if~} j_0\geqslant i_2+2,
    \end{array}
  \right.
\end{align*}
because $k\geqslant 2$ and the equality $j_0=i_2+1-1/k$ is impossible; and

\vskip20pt

\begin{align*}
  ((i_2,j_2,[2))\cdot(i_0,j_0,[0)))\alpha_{[k]} &=
   \left\{
    \begin{array}{ll}
      (i_2-j_2+i_0,j_0,(j_2-i_0+[2))\cap[0))\alpha_{[k]}, & \hbox{if~} j_2<i_0;\\
      (i_2,j_0,[2)\cap[0))\alpha_{[k]},                   & \hbox{if~} j_2=i_0;\\
      (i_2,j_2-i_0+j_0,[2)\cap(i_0-j_2+[0)))\alpha_{[k]}, & \hbox{if~} j_2>i_0
    \end{array}
  \right.
\\
   &=
  \left\{
    \begin{array}{ll}
      (i_2-j_2+i_0,j_0,[0))\alpha_{[k]}, & \hbox{if~} j_2+2\leqslant i_0;\\
      (i_2+1,j_0,[1))\alpha_{[k]},       & \hbox{if~} j_2+1=i_0;\\
      (i_2,j_0,[2))\alpha_{[k]},         & \hbox{if~} j_2=i_0;\\
      (i_2,j_2-i_0+j_0,[2))\alpha_{[k]}, & \hbox{if~} j_2>i_0
    \end{array}
  \right.
\\
   &=
  \left\{
    \begin{array}{ll}
      (k(i_2-j_2+i_0),kj_0,[0)),           & \hbox{if~} j_2+2\leqslant i_0;\\
      (k(i_2+1),kj_0,[1)),                 & \hbox{if~} j_2+1=i_0;\\
      (k(i_2+1)-1,k(j_0+1)-1,[2)),         & \hbox{if~} j_2=i_0;\\
      (k(i{+}1){-}1_2,k(j_2{-}i_0{+}j_0{+}1){-}1,[2)), & \hbox{if~} j_2>i_0,
    \end{array}
  \right.
\end{align*}

\begin{align*}
(i_2,&j_2,[2))\alpha_{[k]}\cdot(i_0,j_0,[0))\alpha_{[k]}=(k(i_2+1)-1,k(j_2+1)-1,[2))\cdot(ki_0,kj_0,[0))  \\
   &=
  \left\{
    \begin{array}{l}
      (k(i_2+1)-1-(k(j_2+1)-1)+ki_0,kj_0,(k(j_2+1)-1-ki_0+[2))\cap[0)), \\
      \qquad \qquad \qquad \qquad \qquad \qquad \qquad \qquad \qquad \qquad \qquad \qquad \quad \hbox{if~} k(j_2{+}1){-}1<ki_0;\\
      (k(i_2+1)-1,kj_0,[2)\cap[0)),
      \qquad \qquad \qquad \qquad \qquad \qquad \; \,
      \hbox{if~} k(j_2{+}1){-}1=ki_0;\\
      (k(i_2+1)-1,k(j_2+1)-1-ki_0+kj_0,[2)\cap(ki_0-(k(j_2+1)-1)+[0))), \\
      \qquad \qquad \qquad \qquad \qquad \qquad \qquad \qquad \qquad \qquad \qquad \qquad \quad \hbox{if~} k(j_2{+}1){-}1>ki_0
    \end{array}
  \right.
   \\
   &=
  \left\{
    \begin{array}{ll}
      (k(i_2-j_2+i_0),kj_0,(k(j_2+1)-1-ki_0+[2))), & \hbox{if~} j_2+1<i_0+1/k;\\
      (k(i_2+1)-1,kj_0,[2)),                       & \hbox{if~} j_2+1=i_0+1/k;\\
      (k(i_2+1)-1,k(j_2-i_0+j_0+1)-1,[2)),         & \hbox{if~} j_2+1>i_0+1/k
    \end{array}
  \right.
   \\
   &=
  \left\{
    \begin{array}{ll}
      (k(i_2-j_2+i_0),kj_0,[0)),                  & \hbox{if~} j_2+2\leqslant i_0;\\
      (k(i_2+1),kj_0,(k(j_2+1)-1-ki_0+[2))), & \hbox{if~} j_2+1=i_0;\\
      (k(i_2+1)-1,k(j_0+1)-1,[2)),         & \hbox{if~} j_2=i_0;\\
      (k(i+1)-1_2,k(j_2-i_0+j_0+1)-1,[2)), & \hbox{if~} j_2>i_0,
    \end{array}
  \right.
\end{align*}
because $k\geqslant 2$ and the equality $j_2+1=i_0+1/k$ is impossible. This completes the proof of the lem\-ma.
\end{proof}

\begin{remark}
Proposition~\ref{proposition-2.5} implies that for any positive integer $k$ the endomorphism $\alpha_{[k]}$ of $\boldsymbol{B}_{\omega}^{\mathscr{F}^3}$ is a extension of the endomorphism $\alpha_{k,0}$  of its subsemigroup $\boldsymbol{B}_{\omega}^{\mathscr{F}_{0,1}}$.
\end{remark}

\begin{proposition}\label{proposition-2.5}
Let $\varepsilon$ be an injective monoid endomorphism of  $\boldsymbol{B}_{\omega}^{\mathscr{F}^3}$ such that
\begin{equation*}
(0,0,[0))\varepsilon=(0,0,[0)), \qquad (0,0,[1))\varepsilon=(0,0,[1)), \qquad \hbox{and} \qquad (0,0,[2))\varepsilon\in \boldsymbol{B}_{\omega}^{\{[2)\}}.
\end{equation*}
Then there exists a positive integer $k$ such that $\varepsilon=\alpha_{[k]}$.
\end{proposition}

\begin{proof}[\textsl{Proof}]
If $(0,0,[2))\varepsilon=(0,0,[2))$ then by Theorem \ref{theorem-2.1} we get that $\varepsilon$ is the identity map of $\boldsymbol{B}_{\omega}^{\mathscr{F}^3}$, and hence $\varepsilon=\alpha_{[k]}$ for $k=1$.

\smallskip

Later we assume that $(0,0,[2))\varepsilon\neq(0,0,[2))$. By Lemma~2 of \cite{Gutik-Pozdniakova=2023} the restrictions of the map $\varepsilon$ onto the subsemigroup $\boldsymbol{B}_{\omega}^{\mathscr{F}_{0,1}}$ of $\boldsymbol{B}_{\omega}^{\mathscr{F}^3}$ is an injective monoid endomorphism of  $\boldsymbol{B}_{\omega}^{\mathscr{F}_{0,1}}$. The above arguments, the assumptions of the proposition, and Theorem~\ref{theorem-2.2} imply that there exists a positive integer $k$ such that
\begin{align*}
  (i,j,[0))\varepsilon&=(ki,kj,[0)), \\
  (i,j,[1))\varepsilon&=(ki,kj,[1)),
\end{align*}
for all $i,j\in\omega$. Hence the restrictions of the endomorphisn $\varepsilon$ onto the subsemigroup $\boldsymbol{B}_{\omega}^{\mathscr{F}_{0,1}}$ of $\boldsymbol{B}_{\omega}^{\mathscr{F}^3}$ coincides with injective monoid endomorphism $\alpha_{k,0}$ of  $\boldsymbol{B}_{\omega}^{\mathscr{F}_{0,1}}$.
Again, by Lemma~2 of \cite{Gutik-Pozdniakova=2023} the restrictions of the map $\varepsilon$ onto the subsemigroup $\boldsymbol{B}_{\omega}^{\mathscr{F}_{1,2}}$ of $\boldsymbol{B}_{\omega}^{\mathscr{F}^3}$ is an injective monoid endomorphism of  $\boldsymbol{B}_{\omega}^{\mathscr{F}_{1,2}}$. This, the above arguments, and Theorem~\ref{theorem-2.2} imply that there exists a positive integer $s\in\{1,\ldots,k-1\}$ such that
\begin{align*}
  (i,j,[2))\varepsilon&=(ki+s,kj+s,[1)),
\end{align*}
for all $i,j\in\omega$.

\smallskip

We claim that $s=k-1$. Indeed, the semigroup operation of $\boldsymbol{B}_{\omega}^{\mathscr{F}^3}$ implies that
\begin{align*}
  (1,1,[0))\cdot (0,0,[2))&=(1,1,[0)\cap(-1+[2)))=\\
                          &=(1,1,[0)\cap([1)))=\\
                          &=(1,1,[1)).
\end{align*}
Since $\varepsilon$ is an endomorphism of $\boldsymbol{B}_{\omega}^{\mathscr{F}^3}$, we get that
\begin{align*}
   (k,k,[1))&=(1,1,[1))\varepsilon= \\
   &=((1,1,[0))\cdot (0,0,[2)))\varepsilon=\\
   &=(1,1,[0))\varepsilon\cdot (0,0,[2))\varepsilon=\\
   &=(k,k,[0))\cdot (s,s,[2))=\\
   &=(k,k-s+s,[0)\cap(s-k+[2)))=\\
   &=(k,k,[0)\cap[s-k+2)),
\end{align*}
which implies that $\max\{0, s-k+2\}=1$. Then $s-k+2=1$, and hence $s=k-1$.
\end{proof}

\begin{proposition}\label{proposition-2.6}
Let $\varepsilon$ be an injective monoid endomorphism of the semigroup $\boldsymbol{B}_{\omega}^{\mathscr{F}^3}$. If $(0,0,[0))\varepsilon=(0,0,[0))$ and $(0,0,[1))\varepsilon=(0,0,[1))$, then $\varepsilon=\alpha_{[k]}$ for some positive inte\-ger~$k$.
\end{proposition}

\begin{proof}
Suppose that $(0,0,[2))\varepsilon\in \boldsymbol{B}_{\omega}^{\{[1)\}}$. Since $(0,0,[0))\varepsilon=(0,0,[0))$ and $(0,0,[1))\varepsilon=(0,0,[1))$, Theorem~\ref{theorem-2.2} implies that there exists a positive integer $k$ such that $(i,j,[0))\varepsilon=(ki,kj,[0))$ and $(i,j,[1))\varepsilon=(ki,kj,[1))$ for all $i,j\in\omega$. Since $(0,0,[2))$ is an idempotent of $\boldsymbol{B}_{\omega}^{\mathscr{F}^3}$, Proposition~1.4.21(2) of \cite{Lawson=1998} implies so is $(0,0,[2))\varepsilon$. By Lemma~2 of \cite{Gutik-Mykhalenych=2020} there exists $s\in\omega$ such that $(0,0,[2))\varepsilon=(s,s,[1))$. The inequalities $(1,1,[1))\preccurlyeq(0,0,[2))\preccurlyeq(0,0,[1))$ and Proposition~1.4.21(6) of \cite{Lawson=1998} imply that
\begin{align*}
  (k,k,[1))&=(1,1,[1))\varepsilon \preccurlyeq\\
           &\preccurlyeq(0,0,[2))\varepsilon=\\
           &=(s,s,[1))\preccurlyeq\\
           &\preccurlyeq(0,0,[1))=\\
           &=(0,0,[1))\varepsilon.
\end{align*}
Since the endomorphism $\varepsilon$ is an injective map, Lemma~5 of \cite{Gutik-Mykhalenych=2020} implies that $0<s<k$. The semigroup operation of $\boldsymbol{B}_{\omega}^{\mathscr{F}^3}$ implies that
\begin{align*}
  (1,1,[0))\cdot (0,0,[2))&=(1,1,[0)\cap(-1+[2)))=\\
                          &=(1,1,[0)\cap([1)))=\\
                          &=(1,1,[1)),
\end{align*}
and hence we get that
\begin{align*}
  (k,k,[1))&=(1,1,[1))\varepsilon= \\
   &=((1,1,[0))\cdot (0,0,[2)))\varepsilon=\\
   &=(1,1,[0))\varepsilon\cdot (0,0,[2))\varepsilon=\\
   &=(s,s,[1))\cdot (k,k,[0))=\\
   &=(s-s+k,k,(s-k+[1))\cap[0))=\\
   &=(k,k,[0)),
\end{align*}
because $s<k$. The obtained contradiction implies that $(0,0,[2))\varepsilon\notin \boldsymbol{B}_{\omega}^{\{[1)\}}$.

\smallskip

Suppose that $(0,0,[2))\varepsilon\in \boldsymbol{B}_{\omega}^{\{[0)\}}$. Since  $(0,0,[2))$ is an idempotent of $\boldsymbol{B}_{\omega}^{\mathscr{F}^3}$,  Proposition~1.4.21(2) of \cite{Lawson=1998} and Lemma~2 of \cite{Gutik-Mykhalenych=2020} imply that there exists $t\in\omega$ such that $(0,0,[2))\varepsilon=(t,t,[0))$. The semigroup operation of $\boldsymbol{B}_{\omega}^{\mathscr{F}^3}$ implies that
\begin{align*}
  (1,1,[0))\cdot (0,0,[2))&=(1,1,[0)\cap(-1+[2)))=\\
                          &=(1,1,[0)\cap([1)))=\\
                          &=(1,1,[1)),
\end{align*}
and by Theorem~\ref{theorem-2.2} we get that there exist a positive integer $k$ such that $(i,j,[0))\varepsilon=(ki,kj,[0))$ and $(i,j,[1))\varepsilon=(ki,kj,[1))$ for all $i,j\in\omega$. Then we have that
\begin{align*}
  (k,k,[1))&=(1,1,[1))\varepsilon= \\
   &=((1,1,[0))\cdot (0,0,[2)))\varepsilon=\\
   &=(1,1,[0))\varepsilon\cdot (0,0,[2))\varepsilon=\\
   &=(t,t,[0))\cdot (k,k,[0))=\\
   &=(\max\{t,k\},\max\{t,k\},[0))\in \boldsymbol{B}_{\omega}^{\{[0)\}},
\end{align*}
a contradiction. Hence $(0,0,[2))\varepsilon\notin \boldsymbol{B}_{\omega}^{\{[0)\}}$.

\smallskip

The above arguments imply that $(0,0,[2))\varepsilon\in \boldsymbol{B}_{\omega}^{\{[2)\}}$. Next we apply Propositi\-on~\ref{proposition-2.5}.
\end{proof}

\begin{proposition}\label{proposition-2.7}
For an arbitrary injective monoid endomorphism $\varepsilon$ of the semigroup $\boldsymbol{B}_{\omega}^{\mathscr{F}^3}$ there exist no a positive integers $k$ and $p\in\{1,\ldots,k-1\}$ such that the restriction $\varepsilon{\downharpoonright}_{\boldsymbol{B}_{\omega}^{\mathscr{F}_{0,1}}}$ of the map $\varepsilon$ onto the subsemigroup $\boldsymbol{B}_{\omega}^{\mathscr{F}_{0,1}}$ of $\boldsymbol{B}_{\omega}^{\mathscr{F}^3}$ coincides with the endomorphism $\alpha_{k,p}$ of $\boldsymbol{B}_{\omega}^{\mathscr{F}_{0,1}}$.
\end{proposition}

\begin{proof}
Suppose to the contrary that exist a positive integer $k$ and $p\in\{1,\ldots,k-1\}$ such that   $\varepsilon{\downharpoonright}_{\boldsymbol{B}_{\omega}^{\mathscr{F}_{0,1}}}=\alpha_{k,p}$. Then we have that
\begin{align*}
  (i,j,[0))\varepsilon&=(ki,kj,[0)), \\
  (i,j,[1))\varepsilon&=(p+ki,p+kj,[1)),
\end{align*}
for all $i,j\in\omega$.

\smallskip

Suppose that $(0,0,[2))\varepsilon\in \boldsymbol{B}_{\omega}^{\{[2)\}}$.
By the choice of the integer $p$ and by the description of the natural partial order on $E(\boldsymbol{B}_{\omega}^{\mathscr{F}^3})$ (see Lemma~5 of \cite{Gutik-Mykhalenych=2020} or Proposition~3 in \cite{Gutik-Mykhalenych=2021}) we get that there exists a positive integer $t$ such that $(0,0,[2))\varepsilon=(t,t,[2))$. The semigroup operation of $\boldsymbol{B}_{\omega}^{\mathscr{F}^3}$ implies that
\begin{equation*}
  (1,1,[0))\cdot (0,0,[2))=(1,1,[1)),
\end{equation*}
and hence we have that
\begin{align*}
  (k,k,[0))\cdot (t,t,[2))&=(1,1,[0))\varepsilon\cdot (0,0,[2))\varepsilon=\\
                          &=(1,1,[1))\varepsilon=\\
                          &=(p+k,p+k,[1)).
\end{align*}
The structure of the natural partial order on $E(\boldsymbol{B}_{\omega}^{\mathscr{F}^3})$ (see Proposition~3 in \cite{Gutik-Mykhalenych=2021}) implies that
\begin{equation*}
(1,1,[1))\preccurlyeq (0,0,[2))\preccurlyeq (0,0,[1)).
\end{equation*}
Hence by Proposition~1.4.21(6) of \cite{Lawson=1998} we have that
\begin{align*}
  (p+k,p+k,[1))&=(1,1,[1))\varepsilon\preccurlyeq \\
  &\preccurlyeq (t,t,[2))=\\
  &=(0,0,[2))\varepsilon\preccurlyeq\\
  &\preccurlyeq(0,0,[1))\varepsilon=\\
  &=(p,p.[1)).
\end{align*}
The above arguments and Lemma~5 of \cite{Gutik-Mykhalenych=2020} imply that $p\leqslant t\leqslant k+p$. Then the equalities
\begin{align*}
  (p+k,p+k,[1))&=(k,k,[0))\cdot (t,t,[2))= \\
   &=
   \left\{
     \begin{array}{ll}
       (t,t,[2)),               & \hbox{if~} k\leqslant t;\\
       (k,k,[0)\cap (t-k+[2))), & \hbox{if~} k>t
     \end{array}
   \right.
\end{align*}
imply that $t-k=-1$ and $k=k+p$. The last equality contradicts the assumption.

\smallskip

Suppose that $(0,0,[2))\varepsilon\in \boldsymbol{B}_{\omega}^{\{[1)\}}$.
Then by the choice of the integer $p$ and by the structure of the natural partial order on $E(\boldsymbol{B}_{\omega}^{\mathscr{F}^3})$ (see Lemma~5 of \cite{Gutik-Mykhalenych=2020} or Proposition~3 in \cite{Gutik-Mykhalenych=2021}) we obtain that there exists a positive integer $t$ such that $(0,0,[2))\varepsilon=(t,t,[1))$. Since
\begin{equation*}
(1,1,[1))\preccurlyeq (0,0,[2))\preccurlyeq (0,0,[1)).
\end{equation*}
by Proposition~1.4.21(6) of \cite{Lawson=1998} we have that
\begin{align*}
  (p+k,p+k,[1))&=(1,1,[1))\varepsilon\preccurlyeq\\
               &\preccurlyeq(t,t,[1))=\\
               &=(0,0,[2))\varepsilon\preccurlyeq\\
               &\preccurlyeq(0,0,[1))\varepsilon=\\
               &=(p,p.[1)).
\end{align*}
The above arguments and Lemma~5 of \cite{Gutik-Mykhalenych=2020} imply that $p\leqslant t\leqslant k+p$. These inequalities and the injectivity of the map $\varepsilon$ imply that $p<t<k+p$. Then the equality
\begin{equation*}
  (1,1,[0))\cdot (0,0,[2))=(1,1,[1)),
\end{equation*}
imply that
\begin{align*}
  (p+k,p+k,[1))&=(k,k,[0))\cdot (t,t,[1))= \\
   &=
   \left\{
     \begin{array}{ll}
       (t,t,[1)), & \hbox{if~} k\leqslant t;\\
       (k,k,[0)), & \hbox{if~} k>t,
     \end{array}
   \right.
\end{align*}
and hence $t=k+p$, a contradiction.

\smallskip

Suppose that $(0,0,[2))\varepsilon\in \boldsymbol{B}_{\omega}^{\{[0)\}}$.
Then by the choice of the integer $p$ and the description of the natural partial order on $E(\boldsymbol{B}_{\omega}^{\mathscr{F}^3})$ (see Lemma~5 of \cite{Gutik-Mykhalenych=2020} or Proposition~3 in \cite{Gutik-Mykhalenych=2021}) we get that there exists a positive integer $t$ such that $(0,0,[2))\varepsilon=(t,t,[0))$. Since
\begin{equation*}
(1,1,[1))\preccurlyeq (0,0,[2))\preccurlyeq (0,0,[1)).
\end{equation*}
by Proposition~1.4.21(6) of \cite{Lawson=1998} we have that
\begin{align*}
  (p+k,p+k,[1))&=(1,1,[1))\varepsilon\preccurlyeq\\
               &\preccurlyeq(t,t,[0))=\\
               &=(0,0,[2))\varepsilon\preccurlyeq\\
               &\preccurlyeq(0,0,[1))\varepsilon=\\
               &=(p,p.[1)).
\end{align*}
The above arguments and Lemma~5 of \cite{Gutik-Mykhalenych=2020} imply that $p\leqslant t\leqslant k+p$. Since
\begin{equation*}
  (1,1,[0))\cdot (0,0,[2))=(1,1,[1)),
\end{equation*}
we obtain that
\begin{align*}
  (p+k,p+k,[1))&=(k,k,[0))\cdot (t,t,[0))=\\
               &=(\max\{k,t\},\max\{k,t\},[0)),
\end{align*}
a contradiction.

\smallskip

The obtained contradictions imply the statement of the proposition.
\end{proof}

\begin{proposition}\label{proposition-2.8}
For any injective monoid endomorphism $\varepsilon$ of the semigroup $\boldsymbol{B}_{\omega}^{\mathscr{F}^3}$ there exist no a positive integers $k\geqslant 2$ and $p\in\{1,\ldots,k-1\}$ such that the restriction $\varepsilon{\downharpoonright}_{\boldsymbol{B}_{\omega}^{\mathscr{F}_{0,1}}}$ of the map $\varepsilon$ onto the subsemigroup $\boldsymbol{B}_{\omega}^{\mathscr{F}_{0,1}}$ of $\boldsymbol{B}_{\omega}^{\mathscr{F}^3}$ coincides with the endomorphism $\beta_{k,p}$ of $\boldsymbol{B}_{\omega}^{\mathscr{F}_{0,1}}$.
\end{proposition}

\begin{proof}
Suppose to the contrary that exist a positive integer $k$ and $p\in\{1,\ldots,k-1\}$ such that   $\varepsilon{\downharpoonright}_{\boldsymbol{B}_{\omega}^{\mathscr{F}_{0,1}}}=\beta_{k,p}$. Then we have that
\begin{align*}
  (i,j,[0))\varepsilon&=(ki,kj,[0)), \\
  (i,j,[1))\varepsilon&=(p+ki,p+kj,[0)),
\end{align*}
for all $i,j\in\omega$.

\smallskip

Suppose that $(0,0,[2))\varepsilon\in \boldsymbol{B}_{\omega}^{\{[2)\}}$.
Then by the choice of the integer $p$ and the description of the natural partial order on $E(\boldsymbol{B}_{\omega}^{\mathscr{F}^3})$ (see Lemma~5 of \cite{Gutik-Mykhalenych=2020} or Proposition~3 in \cite{Gutik-Mykhalenych=2021}) we obtain that there exists a positive integer $t$ such that $(0,0,[2))\varepsilon=(t,t,[2))$. Since
\begin{equation*}
(1,1,[1))\preccurlyeq (0,0,[2))\preccurlyeq (0,0,[1)).
\end{equation*}
by Proposition~1.4.21(6) of \cite{Lawson=1998} we have that
\begin{align*}
  (p+k,p+k,[1))&=(1,1,[1))\varepsilon\preccurlyeq\\
               &\preccurlyeq(t,t,[2))=\\
               &=(0,0,[2))\varepsilon\preccurlyeq\\
               &\preccurlyeq(0,0,[1))\varepsilon=\\
               &=(p,p.[1)).
\end{align*}
The above arguments and Lemma~5 of \cite{Gutik-Mykhalenych=2020} imply that $p\leqslant t\leqslant k+p$. The semigroup operation of $\boldsymbol{B}_{\omega}^{\mathscr{F}^3}$ implies that
\begin{equation*}
  (1,1,[0))\cdot (0,0,[2))=(1,1,[1)),
\end{equation*}
and hence we have that
\begin{align*}
  (k,k,[0))\cdot (t,t,[2))&=(1,1,[0))\varepsilon\cdot (0,0,[2))\varepsilon=\\
                          &=(1,1,[1))\varepsilon=\\
                          &=(p+k,p+k,[0)).
\end{align*}
Then the equalities
\begin{align*}
  (p+k,p+k,[0))&=(k,k,[0))\cdot (t,t,[2))= \\
   &=
   \left\{
     \begin{array}{ll}
       (t,t,[2)),               & \hbox{if~} k\leqslant t;\\
       (k,k,[0)\cap (t-k+[2))), & \hbox{if~} k>t
     \end{array}
   \right.
\end{align*}
imply that $k=k+p$, and hence $p=0$. A contradiction.

\smallskip

Suppose that $(0,0,[2))\varepsilon\in \boldsymbol{B}_{\omega}^{\{[1)\}}$.
The choice of the integer $p$ and the structure of the natural partial order on $E(\boldsymbol{B}_{\omega}^{\mathscr{F}^3})$ (see Lemma~5 of \cite{Gutik-Mykhalenych=2020} or Proposition~3 in \cite{Gutik-Mykhalenych=2021}) imply that there exists a positive integer $t$ such that $(0,0,[2))\varepsilon=(t,t,[1))$. Similar as in the previous case we get that $p\leqslant t\leqslant k+p$. Then the equality
\begin{equation*}
  (1,1,[0))\cdot (0,0,[2))=(1,1,[1)),
\end{equation*}
implies that
\begin{align*}
  (k,k,[0))\cdot (t,t,[1))&=(1,1,[0))\varepsilon\cdot (0,0,[2))\varepsilon=\\
                          &=(1,1,[1))\varepsilon=\\
                          &=(p+k,p+k,[0)),
\end{align*}
and hence the equalities
\begin{align*}
  (p+k,p+k,[0))&=(k,k,[0))\cdot (t,t,[1))= \\
   &=
   \left\{
     \begin{array}{ll}
       (t,t,[1)), & \hbox{if~} k\leqslant t;\\
       (k,k,[0)), & \hbox{if~} k>t
     \end{array}
   \right.
\end{align*}
imply that $k=k+p$, and hence $p=0$. A contradiction.

\smallskip

Suppose that $(0,0,[2))\varepsilon\in \boldsymbol{B}_{\omega}^{\{[0)\}}$.
The choice of the integer $p$ and the structure of the  natural partial order on $E(\boldsymbol{B}_{\omega}^{\mathscr{F}^3})$ (see Lemma~5 of \cite{Gutik-Mykhalenych=2020} or Proposition~3 in \cite{Gutik-Mykhalenych=2021}) imply that there exists a positive integer $t$ such that $(0,0,[2))\varepsilon=(t,t,[0))$. Similar as in the previous case we get that $p\leqslant t\leqslant k+p$. Then the equality
\begin{equation*}
  (1,1,[0))\cdot (0,0,[2))=(1,1,[1)),
\end{equation*}
implies that
\begin{align*}
  (k,k,[0))\cdot (t,t,[0))&=(1,1,[0))\varepsilon\cdot (0,0,[2))\varepsilon=\\
                          &=(1,1,[1))\varepsilon=\\
                          &=(p+k,p+k,[0)).
\end{align*}
Then we have that
\begin{align*}
  (p+k,p+k,[0))&=(k,k,[0))\cdot (t,t,[0))= \\
   &=
   \left\{
     \begin{array}{ll}
       (t,t,[0)), & \hbox{if~} k\leqslant t;\\
       (k,k,[0)), & \hbox{if~} k>t.
     \end{array}
   \right.
\end{align*}
If $k=k+p$ then $p=0$, which contradicts the assumption of the proposition. If $t=p+k$ then
\begin{equation*}
  (1,1,[1))\varepsilon=(p+k,p+k,[0))= (0,0,[2))\varepsilon,
\end{equation*}
which contradicts the injectivity of the map $\varepsilon$.

\smallskip

The obtained contradictions imply the statement of the proposition.
\end{proof}

The following theorem summarises the main result of this section and it follows from Lemma \ref{lemma-2.4} and Propositions \ref{proposition-2.5}--\ref{proposition-2.8}.

\begin{theorem}\label{theorem-2.9}
Let $\mathscr{F}^3=\{[0),[1),[2)\}$ and $\varepsilon$ be an injective monoid endomorphism of the semigroup $\boldsymbol{B}_{\omega}^{\mathscr{F}^3}$. If the restriction $\varepsilon{\downharpoonright}_{\boldsymbol{B}_{\omega}^{\mathscr{F}_{0,1}}}$ of the map $\varepsilon$ onto the subsemigroup $\boldsymbol{B}_{\omega}^{\mathscr{F}_{0,1}}$ of $\boldsymbol{B}_{\omega}^{\mathscr{F}^3}$ is an injective monoid endomorphism of $\boldsymbol{B}_{\omega}^{\mathscr{F}_{0,1}}$, then $\varepsilon=\alpha_{[k]}$ for some positive integer $k$.
\end{theorem}

\begin{theorem}\label{theorem-3.1}
Let $\mathscr{F}^3=\{[0),[1),[2)\}$. Every injective monoid endomorphism of the semigroup $\boldsymbol{B}_{\omega}^{\mathscr{F}^3}$ is an extension of injective endomorphisms of its submonoid $\boldsymbol{B}_{\omega}^{\mathscr{F}_{0,1}}$.
\end{theorem}

\begin{proof}
Suppose to the contrary that there exists an injective monoid endomorphism $\varepsilon$ of the semigroup $\boldsymbol{B}_{\omega}^{\mathscr{F}^3}$ such that the restriction $\varepsilon{\downharpoonright}_{\boldsymbol{B}_{\omega}^{\mathscr{F}_{0,1}}}$ of the map $\varepsilon$ onto the subsemigroup $\boldsymbol{B}_{\omega}^{\mathscr{F}_{0,1}}$ of $\boldsymbol{B}_{\omega}^{\mathscr{F}^3}$ is not a monoid endomorphism of $\boldsymbol{B}_{\omega}^{\mathscr{F}_{0,1}}$. By Proposition~3 of \cite{Gutik-Mykhalenych=2020}, for any $n=0,1,2$ the semigroup  $\boldsymbol{B}_{\omega}^{\{[n)\}}$ is isomorphic to the bicyclic semigroup.  By Proposition~4 of \cite{Gutik-Mykhalenych=2021} we have that $(i,j,[0))\varepsilon\in \boldsymbol{B}_{\omega}^{\{[0)\}}$ for all $i,j\in\omega$, because $\varepsilon$ is an injective monoid endomorphism of the semigroup $\boldsymbol{B}_{\omega}^{\mathscr{F}^3}$. Moreover,  by Theorem~1 from \cite{Gutik-Prokhorenkova-Sekh=2021} there exists a positive integer $k$ such that $(i,j,[0))\varepsilon=(ki,kj,[0))$ for all $i,j\in\omega$. Again, Proposition~4 of \cite{Gutik-Mykhalenych=2021} implies that for any $n\in\{1,2\}$ there exists $m_m\in\{0,1,2\}$ such that $(i,j,[n))\varepsilon\in \boldsymbol{B}_{\omega}^{\{[m_n)\}}$ for all $i,j\in\omega$. The above arguments and Theorem~\ref{theorem-2.2} imply that $(i,j,[1))\varepsilon\in \boldsymbol{B}_{\omega}^{\{[2)\}}$ for all $i,j\in\omega$.

\smallskip

We remark that the assumption that
\begin{equation*}
  (i,j,[2))\varepsilon\in \boldsymbol{B}_{\omega}^{\mathscr{F}_{0,1}}, \qquad \hbox{for all} \quad i,j\in\omega,
\end{equation*}
contradicts the equality
\begin{equation*}
  (1,1,[0))\cdot (0,0,[2))=(1,1,[1)).
\end{equation*}
By Proposition~1.4.21(2) of \cite{Lawson=1998}, $(0,0,[2))\varepsilon$ is an idempotent of $\boldsymbol{B}_{\omega}^{\mathscr{F}^3}$. If $(0,0,[2))\varepsilon=(t,t,[0))$ for some $t\in\omega$ (see Lemma~2 in \cite{Gutik-Mykhalenych=2020}), then we have that
\begin{align*}
  (1,1,[1))\varepsilon &=((1,1,[0))\cdot (0,0,[2)))\varepsilon= \\
   &=(1,1,[0))\varepsilon\cdot(0,0,[2))\varepsilon= \\
   &=(k,k,[0))\cdot (t,t,[0))=\\
   &=(\max\{k,t\},\max\{k,t\},[0))\in \boldsymbol{B}_{\omega}^{\{[0)\}}.
\end{align*}
This contradicts the condition that $(i,j,[1))\varepsilon\in \boldsymbol{B}_{\omega}^{\{[2)\}}$ for all $i,j\in\omega$. If  $(0,0,[2))\varepsilon=(t,t,[1))$ for some $t\in\omega$ (see Lemma~2 in \cite{Gutik-Mykhalenych=2020}), then we obtain that
\begin{align*}
  (1,1,[1))\varepsilon &=((1,1,[0))\cdot (0,0,[2)))\varepsilon= \\
   &=(1,1,[0))\varepsilon\cdot(0,0,[2))\varepsilon= \\
   &=(k,k,[0))\cdot (t,t,[1))=\\
   &=
   \left\{
     \begin{array}{ll}
       (t,t,[1)), & \hbox{if~} t\geqslant k; \\
       (k,k,[0)), & \hbox{if~} t<k.
     \end{array}
   \right.
\end{align*}
This contradicts the condition that $(i,j,[1))\varepsilon\in \boldsymbol{B}_{\omega}^{\{[2)\}}$ for all $i,j\in\omega$.

\smallskip

Suppose that $(0,0,[2))\varepsilon\in \boldsymbol{B}_{\omega}^{\{[2)\}}$. By Lemma~2 from \cite{Gutik-Mykhalenych=2020} there exists $t\in\omega$ such that $(0,0,[2))\varepsilon=(t,t,[2))$. Since $(0,0,[2))\preccurlyeq(0,0,[1))$, Proposition~1.4.21(6) of \cite{Lawson=1998} implies that $(0,0,[2))\varepsilon\preccurlyeq(0,0,[1))\varepsilon$. If $(0,0,[2))\varepsilon=(0,0,[2))$, then by the equality $(0,0,[0))\varepsilon=(0,0,[0))$ and
\begin{align*}
(0,0,[2))&=(0,0,[2))\varepsilon\preccurlyeq\\
         &\preccurlyeq(0,0,[1))\varepsilon\preccurlyeq\\
         &\preccurlyeq(0,0,[0))\varepsilon=\\
         &=(0,0,[0))
\end{align*}
we obtain that $(0,0,[1))\varepsilon=(0,0,[1))$. Theorem~\ref{theorem-2.1} implies that $\varepsilon$ is the identity map of $\boldsymbol{B}_{\omega}^{\mathscr{F}^3}$, which contradicts the assumption. Hence we have that $t\neq 0$.

\smallskip

Suppose that $(0,0,[1))\varepsilon=(p,p,[2))$ for some $p\in\omega$.
Since
\begin{align*}
  (1,0,[0))\cdot(0,0,[1))\cdot(0,1,[0))&=((1,0,[1))\cdot(0,1,[0))=\\
                                       &=(1,1,[1)),
\end{align*}
we have that
\begin{align*}
  (1,1,[1))\varepsilon& =((1,0,[0))\cdot(0,0,[1))\cdot(0,1,[0)))\varepsilon= \\
   &=(1,0,[0))\varepsilon\cdot(0,0,[1))\varepsilon\cdot(0,1,[0))\varepsilon=\\
   &=(k,0,[0))\cdot(p,p,[2))\cdot(0,k,[0))= \\
   &=(k+p,p,[2))\cdot(0,k,[0))= \\
   &=(k+p,k+p,[2)).
\end{align*}
Put $(0,1,[1))\varepsilon=(x,y,[2))$. By Proposition 1.4.21 from \cite{Lawson=1998} and Lemma~4 of \cite{Gutik-Mykhalenych=2020} we get that
\begin{align*}
  (1,0,[1))\varepsilon&=((0,1,[1))^{-1})\varepsilon=\\
   &=((0,1,[1))\varepsilon)^{-1}=\\
   &=(x,y,[2))^{-1}=\\
   &=(y,x,[2)).
\end{align*}
This implies that
\begin{align*}
  (p,p,[2))&=(0,0,[1))\varepsilon=\\
   &=((0,1,[1))\cdot(1,0,[1)))\varepsilon=\\
   &=(0,1,[1))\varepsilon\cdot(1,0,[1))\varepsilon=\\
   &=(x,y,[2))\cdot(y,x,[2))=  \\
   &=(x,x,[2))
\end{align*}
and
\begin{align*}
  (k+p,k+p,[2))&=(1,1,[1))\varepsilon=\\
   &=((1,0,[1))\cdot(0,1,[1)))\varepsilon=\\
   &=(1,0,[1))\varepsilon\cdot(0,1,[1))\varepsilon= \\
   &=(y,x,[2))\cdot(x,y,[2))=\\
   &=(y,y,[2)).
\end{align*}
Hence by the definition of the semigroup $\boldsymbol{B}_{\omega}^{\mathscr{F}}$ we get that
\begin{equation*}
  (0,1,[1))\varepsilon=(p,k+p,[2)) \qquad \hbox{and} \qquad (1,0,[1))\varepsilon=(k+p,p,[2)).
\end{equation*}
Then for any $i,j\in\omega$ we have that
\begin{align*}
  (i,j,[1))\varepsilon&=((i,0,[1))\cdot(0,j,[1)))\varepsilon=\\
   &=((1,0,[1))^{i}\cdot(0,1,[1))^{j})\varepsilon=\\
   &=((1,0,[1))\varepsilon)^{i}\cdot((0,1,[1))\varepsilon)^{j}= \\
   &=(k+p,p,[2))^{i}\cdot(p,k+p,[2))^{j}=\\
   &=(ki+p,p,[2))\cdot(p,kj+p,[2))=\\
   &=(ki+p,kj+p,[2)).
\end{align*}


 Since $(1,1,[0))\preccurlyeq(0,0,[1))$ in $E\big(\boldsymbol{B}_{\omega}^{\mathscr{F}^3}\big)$, by Proposition 1.4.21(6) from \cite{Lawson=1998} we have that
\begin{equation*}
  (k,k,[0))=(1,1,[0))\varepsilon\preccurlyeq(0,0,[1))\varepsilon=(p,p,[2)).
\end{equation*}
Then Lemma~5 of \cite{Gutik-Mykhalenych=2020} implies that $k\geqslant 2$. Also, the inequalities
\begin{equation*}
  (1,1,[1))\preccurlyeq(0,0,[2))\preccurlyeq(0,0,[1))
\end{equation*}
in $E\big(\boldsymbol{B}_{\omega}^{\mathscr{F}^3}\big)$ and Proposition 1.4.21(6) of \cite{Lawson=1998} imply that
\begin{align*}
  (k+p,k+p,[2))&=(1,1,[1))\varepsilon\preccurlyeq\\
               &\preccurlyeq(0,0,[2))\varepsilon=\\
               &=(t,t,[2))\preccurlyeq\\
               &\preccurlyeq(0,0,[1))\varepsilon=\\
               &=(p,p,[2)).
\end{align*}
By Lemma~5 of \cite{Gutik-Mykhalenych=2020} we get that $p\leqslant t\leqslant k+p$. Since $\varepsilon$ is an injective monoid endomorphism of the semigroup $\boldsymbol{B}_{\omega}^{\mathscr{F}^3}$ we conclude that $p< t< k+p$.

\smallskip

The equality
\begin{equation*}
  (1,1,[0))\cdot (0,0,[2))=(1,1,[1)).
\end{equation*}
implies that
\begin{align*}
  (k+p,k+p,[2))&=(1,1,[1))\varepsilon =\\
   &=((1,1,[0))\cdot (0,0,[2)))\varepsilon= \\
   &=(1,1,[0))\varepsilon\cdot(0,0,[2))\varepsilon= \\
   &=(k,k,[0))\cdot (t,t,[2))=\\
   &=
   \left\{
     \begin{array}{ll}
       (t,t,[2)), & \hbox{if~} k\leqslant t; \\
       (k,k,[1)), & \hbox{if~} k=t+1;\\
       (k,k,[0)), & \hbox{if~} k\geqslant t+2.
     \end{array}
   \right.
\end{align*}
Hence $k\leqslant t$ and $k+p=t$. The last equality implies that
\begin{equation*}
  (1,1,[1))\varepsilon=(k+p,k+p,[2))=(0,0,[2))\varepsilon,
\end{equation*}
which contradicts the injectivity of the map $\varepsilon$.

\smallskip

The obtained contradictions imply the statement of the theorem.
\end{proof}

\section{\textbf{On the monoid of all injective endomorphisms of the semigroup $\boldsymbol{B}_{\omega}^{\mathscr{F}^3}$}}\label{section-4}

Theorems~\ref{theorem-2.9} and \ref{theorem-3.1} imply the following theorem.

\begin{theorem}\label{theorem-4.1}
Let $\mathscr{F}^3=\{[0),[1),[2)\}$ and $\varepsilon$ be an injective monoid endomorphism of the semigroup $\boldsymbol{B}_{\omega}^{\mathscr{F}^3}$. Then $\varepsilon=\alpha_{[k]}$ for some positive integer $k$.
\end{theorem}

By $(\mathbb{N},\cdot)$ we denote the multiplicative semigroup of positive integers.

\begin{theorem}\label{theorem-4.2}
Let $\mathscr{F}^3=\{[0),[1),[2)\}$. Then the monoid $\boldsymbol{End}_*^1(\boldsymbol{B}_{\omega}^{\mathscr{F}^3})$ of all injective endomorphisms of the semigroup $\boldsymbol{B}_{\omega}^{\mathscr{F}^3}$ is isomorphic to $(\mathbb{N},\cdot)$.
\end{theorem}

\begin{proof}
Fix arbitrary injective endomorphisms $\varepsilon_1$ and $\varepsilon_2$ of the semigroup $\boldsymbol{B}_{\omega}^{\mathscr{F}}$. By Theorem~\ref{theorem-4.1} there exist positive integers $k_1$ and $k_2$ such that $\varepsilon_1=\alpha_{[k_1]}$ and $\varepsilon_2=\alpha_{[k_2]}$. Then we have that
\begin{align*}
  ((i,j,[0))\alpha_{[k_1]})\alpha_{[k_2]}&=(k_1i,k_1j,[0))\alpha_{[k_2]}=\\
                                         &=(k_2k_1i,k_2k_1j,[0))=\\
                                         &=(i,j,[0))\alpha_{[k_1\cdot k_2]}; \\ \\
  ((i,j,[1))\alpha_{[k_1]})\alpha_{[k_2]}&=(k_1i,k_1j,[1))\alpha_{[k_2]}=\\
                                         &=(k_2k_1i,k_2k_1j,[1))=\\
                                         &=(i,j,[1))\alpha_{[k_1\cdot k_2]};
\end{align*}
and
\begin{align*}
  ((i,j,[2))\alpha_{[k_1]})\alpha_{[k_2]}&=(k_1(i+1)-1,k_1(j+1)-1,[2))\alpha_{[k_2]}=\\
    &=(k_2(k_1(i+1)-1+1)-1,k_2(k_1(j+1)-1+1)-1,[2))=\\
    &=(k_2k_1(i+1)-1,k_2k_1(j+1)-1,[2))=\\
    &=(i,j,[2))\alpha_{[k_1\cdot k_2]},
\end{align*}
for any $i,j\in\omega$. Hence we obtain that $\alpha_{[k_1]}\alpha_{[k_2]}=\alpha_{[k_1\cdot k_2]}$. It is obvious that the mapping $\mathfrak{i}\colon (\mathbb{N},\cdot)\to \boldsymbol{End}_*^1(\boldsymbol{B}_{\omega}^{\mathscr{F}})$, $k\mapsto \alpha_{[k]}$, is an injective homomorphism and by Theorem~\ref{theorem-4.1} it is surjective.
\end{proof}

\section*{\textbf{Acknowledgements}}

The authors acknowledge the referee for his/her comments and suggestions.



\end{document}